\documentclass[a4paper]{amsart}

\usepackage{color}
\usepackage{amsxtra}
\usepackage{amssymb}
\usepackage{mathtools}
\usepackage{enumerate}
\usepackage{nicefrac}
\mathtoolsset{showonlyrefs} 
\usepackage{ mathrsfs }
\usepackage{tikz}
\usetikzlibrary{positioning, calc}
\usetikzlibrary{shadings}
\usepackage{color}
\usepackage{hyperref}
\hypersetup{
  colorlinks   = true, 
  urlcolor     = blue, 
  linkcolor    = blue, 
  citecolor   = red 
}

\usepackage{selinput}
\SelectInputEncodingList{utf8,latin1}

\SelectInputMappings{
adieresis={ä},
germandbls={ß},
}

\newtheorem{theorem}{Theorem}[section]
\newtheorem{lemma}{Lemma}[section]

\newtheorem{remark}{Remark}[section]

\begin{document}

\title[Non-simultaneous blow-up]{Non-simultaneous blow-up for a system with local and non-local diffusion}


\author[L. M. Del Pezzo]{Leandro M. Del Pezzo}
	\address{Leandro M. Del Pezzo \hfill\break\indent
		CONICET and UTDT \hfill\break\indent
		Departamento de Matem\'aticas y
		Estad\'istica
		\hfill\break\indent Universidad Torcuato Di Tella
		\hfill\break\indent Av. Figueroa Alcorta 7350 (C1428BCW)
		\hfill\break\indent Buenos Aires, ARGENTINA. }
	\email{ldelpezzo@utdt.edu}
	\urladdr{http://cms.dm.uba.ar/Members/ldpezzo/}

\author[R. Ferreira]{Ra\'ul Ferreira}
\address{Ra\'ul Ferreira\hfill\break\indent
Departamento de Matem\'atica Aplicada,
\hfill\break\indent Fac. de C.C. Qu\'{\i}micas, U. Complutense de Madrid,
\hfill\break\indent 28040,  Madrid, Spain.
 }
\email{raul\_ferreira@mat.ucm.es}

\begin{abstract}
    We study the possibility of non-simultaneous blow-up for positive solutions of a
    coupled system of two semilinear equations,
    $u_t = J*u-u+ u^\alpha v^p$, $v_t =\Delta v^+u^qv^\beta$, $p, q, \alpha, \beta>0$ with
    homogeneous Dirichlet boundary conditions and positive initial data.
    We also give the blow-up rates for non-simultaneous blow-up.
    
    \smallskip
    

\end{abstract}

\maketitle

\section{Introduction}\label{intro}	

    Let $\Omega$ be a smooth boundary domain of $\mathbb{R}^N,$
    $p,q>0$ and $\alpha,\beta\ge0.$ Our main interest lies in the blow-up phenomena
    for the following local-nonlocal  reaction
    diffusion system
    \begin{equation}\label{rds}
        \begin{cases}
		u_t=\mathcal{L} u+u^\alpha v^p \qquad & x\in\Omega, t>0,\\
		v_t=\Delta v+u^qv^\beta& x\in\Omega, t>0,
	\end{cases}
    \end{equation}
    under Dirichlet ``boundary" conditions
    \begin{equation}\label{dbc}
	\begin{cases}
	    u(x,t)=0  & x\in\mathbb R^n\setminus \Omega, t>0,\\
	    v(x,t)=0& x\in\partial\Omega, t>0,
	\end{cases}
    \end{equation}
    and initial data
    \begin{equation}\label{id}
	\begin{cases}
	    u(x,0)=u_0(x)  & x\in\Omega,\\
	    v(x,0)=v_0(x)& x\in\Omega.
	 \end{cases}
    \end{equation}
    Here $\mathcal L u=J*u-u$ is a non-local diffusion operator for which the convolution kernel $J$ satisfies the following assumptions which will be assumed
    through out all this article
    \[
	J\in C_0(\mathbb{R}^N,\mathbb{R}) \text{ is a nonnegative function with } J(0)>0 \text{ and }
	\int_{\mathbb{R}^N}J(x)\, dx=1.
    \]
    Regarding the initial data, we assume that $u_0,v_0\in C(\overline{\Omega})$ are non-negative functions.
	
    \medskip

    This type of system appears in some biological models, more precisely
    in the last year, the study of dispersal strategies and the comparison between local and nonlocal
    diffusive behaviours have  attracted a great attention both in terms of experiments and
    from the purely mathematical point of view. It is worth mentioning that, the scientific community
    has not arrive to agreement about  which non-local models better describe the dispersion of
    a real biological population, for this reason many non-local terms have been taken into account
    in order to comprise long-range effects. But the present literature on the subject
    of non-local dispersal mostly considers convolution operators. See for instance \cite{D,K2010,K2012,C,Cosner,M}
    and the references therein for more details.

    \medskip
	
    Problem \eqref{rds}-\eqref{id} provide a simple example of a reaction-diffusion system of
    Fujita type with two different difusión, local in one component and no-local in the other.
    When the two diffusion are given by the same diffusion
    \begin{equation}\label{eq.sistemalocal}
        \left\{
            \begin{array}{ll}
                u_t=\Delta u+u^\alpha v^p\\
                v_t=\Delta v+u^q v^\beta\\
            \end{array}
        \right.
        \qquad
            \left\{
            \begin{array}{ll}
                u_t=\mathcal{L}u+u^\alpha v^p,\\
                v_t=\mathcal{L}v+u^q v^\beta,\\
            \end{array}\right.
    \end{equation}
    under homogeneous boundary conditions, it is well known that there exists blow-up solutions if and only if
    $$
       \alpha>1, \quad \beta>1,\quad \text{and}\quad pq>(1-\alpha)(1-\beta),
    $$
    see for instance \cite{EL,Chen,W,B} for the local case and \cite{RF} for the nonlocal case.  At time $T$ we have,
    $$
        \limsup_{t\to T}\ \max \{\|u(\cdot,t)\|_\infty, \|v(\cdot,t)\|_\infty\}=\infty.
    $$
    Note that the solution blows up if only one of the components blows up, it is what we call non-simultaneous blow-up.
    In \cite{QR, RF} the authors consider the systems \eqref{eq.sistemalocal} and they prove that if only $v$ blows up then $\beta>1+p$.
    The fact that $\beta>1$ ensures the existence of initial data such that $v$ blows up and the condition $p<\beta-1$
    implies that the coupling of the system is weak enough to get non-simultaneous blow-up.

    As previous paper in the study of blow-up with problems that mix local and non-local diffusions, we refer  to \cite{BR} where the the system
    $$
        \left\{
        \begin{array}{l}
            u_t=\mathcal{L}u+v^p \\
            v_t=\Delta v+u^q
        \end{array}\right.
    $$
    has been analysed for mixed boundary condition
    (Dirichlet for the non-local component and Neumann for the local one).
    However, this system is weakly coupled ($\alpha=0$ and $\beta=0$), then non-simultaneous blow-up cannot occurs.

    \medskip

    Existence of a solution can easily be achieved, but uniqueness is subtle due to
    the fact that the parameters $\alpha, \beta, p, q$ can be less than 1. The difficulty comes from the non-Lipschitz character of the reaction, see for instance \cite{EH}. Nevertheless, in this case we can prove existence of a maximal solution just taking the limit
    $$
    (u,v)=\lim_{\varepsilon\to0} (u_\varepsilon, v_\varepsilon)
    $$
    where $(u_\varepsilon, v_\varepsilon)$ is the solution to \eqref{rds} with initial data $(u_\varepsilon(x,0), v_\varepsilon(x,0))=(u_0(x)+\varepsilon, v_0(x)+\varepsilon)$, boundary data $u=\varepsilon$ outside $\Omega$ and $v_\varepsilon=\varepsilon$ at the boundary of $\Omega$ and the non-Lipschitz power replace by
    $$
    f_\varepsilon (w,\gamma)=\left\{
    \begin{array}{ll}
    w^\gamma \qquad & w\ge \varepsilon,\\
    \varepsilon^{\gamma} & w\le \varepsilon.
    \end{array}\right.
    $$
    Moreover, a comparison principle among maximal solutions can easily be obtained, see for instance \cite[Lemma 24]{BR}. Notice that if the reaction terms are Lipschitz we have uniqueness. From now on, we only consider maximal solutions.

    \subsection*{Main results}
    The goal of this paper is to study the non-simultaneous blow-up when we mix local and non-local diffusions.
    Our first goal in the study of system \eqref{rds} is to obtain, in term of the parameters, the existence of blow-up solutions.
    \begin{theorem}\label{teo.1.1}
        Let $(u,v)$ be a  solution of \eqref{rds}-\eqref{id}. Then,
        \begin{enumerate}[(i)]
            \item If $\alpha<1$, $\beta<1$ and $pq\le (1-\alpha)(1-\beta)$ all solutions are global.
            \item Other wise there exist both, blow-up and global solutions.
        \end{enumerate}
    \end{theorem}

    Now, we look for the possibility that just one component of the system \eqref{rds} blows up.
    To see this we restrict ourselves to solutions increasing in time.

    \begin{theorem}\label{teo.non2}
        Let $(u,v)$ be a  solution of \eqref{rds}-\eqref{id} with
        \[
            \begin{array}{l}
                \mathcal L u_0+ u_0^\alpha v_0^p>0,\\
                \Delta v_0 + u_0^q v_0^\beta>0.
            \end{array}
        \]
        \begin{enumerate}[(i)]
                \item If  $u$ blows up while $v$ remains bounded, then
                    \[
                        \max_{x\in\overline\Omega} u(x,t)\sim  (T-t)^{\frac{-1}{\alpha-1}}.
                    \]
                \item If  $v$ blows up while $u$ remains bounded, then $\beta>1+p$ and
                    \[
                        \max_{x\in\overline\Omega} v(x,t)\sim  (T-t)^{\frac{-1}{\beta-1}}.
                    \]
        \end{enumerate}
    \end{theorem}

    To conclude this introduction, we would like to point out that working
    with mixed local and non-local diffusion it turned out to be a not easy
    mathematical problem since the known tools were not adapted in a trivial way.

\subsection*{The rest of this manuscript is organized as follows} In Section \ref{GE},
    we prove that there exist global solutions (the first item of Theorem  \ref{teo.1.1}).
    Then, in Section \ref{BUS}, we prove that there is a blow-up solution
    (the second item of Theorem  \ref{teo.1.1}). Lastly, in Section \ref{nsu},
    we study the possibility that one component of the system \eqref{rds} blows up at time $T<\infty$
    while the other remains bounded (proof of Theorem \ref{teo.non2}).

\section{Global existence}\label{GE}
    In this section, we prove the first item of Theorem \ref{teo.1.1}. To this
    end, we borrow ideas of \cite[Lemma 2.1]{RF}.
	
    \begin{lemma}
	If $p,q>0$ then there exist global solutions.
	Moreover
	\begin{itemize}
	    \item If $\max\{\alpha,\beta\}<1$ and $\kappa=pq-(1-\alpha)(1-\beta)<0$
			then every solution is bounded.
	    \item If $\max\{\alpha,\beta\}<1$ and $\kappa=0$
			then every solution is global.
	\end{itemize}
    \end{lemma}
	
    \begin{proof}
	Let $w,z$ be the positive bounded solutions of
	\[
	    \begin{cases}
		-\mathcal{L}(u)=1 &x\in\Omega,\\
				u=0& x\in\mathbb{R}^N\setminus\Omega,
	    \end{cases}
	\]
	and
	\[
	    \begin{cases}
		-\Delta u=1 &x\in\Omega,\\
		u=0& x\in\partial\Omega,
	    \end{cases}
	\]
	respectively. We next set
	\[
	    \bar{u}=Aw \quad\text{and}\quad\bar{v}=Bz
	\]
	where $A$ and $B$ are positive constants that will be selected so that
	the pair $(\bar{u},\bar{v})$ will be a supersolution of \eqref{rds}.
	Specifically $(\bar{u}, \bar{v})$ is a supersolution of \eqref{rds} if
	\begin{equation}\label{eq:aux1}
	    1\ge A^{\alpha-1} B^p\|w\|_\infty^{\alpha}\|z\|_\infty^{p}
			\quad\text{and}\quad1\ge A^{q} B^{\beta-1}\|w\|_\infty^{q}\|z\|_\infty^{\beta}.
	\end{equation}
	Taking
	\[
	    1= A^{\alpha-1} B^p\|w\|_\infty^{\alpha}\|z\|_\infty^{p},
	\]
	that is
	\[
	    B=\dfrac{A^{\nicefrac{(1-\alpha)}p}}{\|w\|_\infty^{\nicefrac{\alpha}p}\|z\|_\infty},
	\]
	the second inequality of \eqref{eq:aux1} reads
	\[
	    A^{\nicefrac{\kappa}p}\le\|w\|_{\infty}^{-\nicefrac{(pq+\alpha(1-\beta))}p}
			\|z\|_{\infty}^{-1}
	\]
	where $\kappa=pq-(1-\alpha)(1-\beta)\neq0.$
	Therefore 		
	\begin{itemize}
		\item If $\kappa>0$ then the inequalities \eqref{eq:aux1} hold for $A$ small.
		\item If $\kappa<0$ then the inequalities \eqref{eq:aux1} hold for $A$ large.
	\end{itemize}
	Thus, in both cases, we obtain global solution by comparison.
		
	Moreover if $\kappa<0$ and $\max\{\alpha,\beta\}<1,$ the inequalities \eqref{eq:aux1}
		hold for $A$ and $B$ large. For this reason, in this case, every solution is bounded.
		
	We now study the case $k=0$ and $\min\{\alpha,\beta\}\ge1.$ In this case, we take
	\[
	    B=\dfrac{1}{\|w\|_\infty^{\nicefrac{\alpha}p}\|z\|_\infty}
	\]
	and the inequalities \eqref{eq:aux1} read
	\[
	    1\ge A^{\alpha-1}
	    \quad\text{and}\quad \|w\|_{\infty}^{-\nicefrac{(pq+\alpha(1-\beta))}p}
	    \|z\|_{\infty}^{-1}\ge A^{q}.
	\]
	Therefore the the inequalities \eqref{eq:aux1} hold for $A$ small.
	Again, we obtain global solution by comparison.
		
	Finally, we study the case $\kappa=0,$ and $\max\{\alpha,\beta\}<1.$
	We propose as a supersolution of
	\eqref{rds} the pair
	\[
        	    \bar{u}(x,t)=Ae^{Bt} \quad\text{and}\quad\bar{v}(x,t)=Ce^{Dt}
	\]
	where $A,B,C,D$ are positive constants such that
	\[
	    Be^{((1-\alpha)B-pD)t}\ge A^{\alpha-1}C^p \quad\text{ and }\quad
		De^{((1-\beta)B-qD)t}\ge A^qC^{1-\beta}.
	\]
		
	If we take $D=B\tfrac{(1-\alpha)}{p},$
	the last inequalities reads
	\[
	    B\ge A^{\alpha-1}C^p \quad\text{ and }\quad
			\frac{(1-\alpha)}{p}B\ge A^qC^{1-\beta}.
	\]
	since $\kappa=0.$ Therefore the previous inequalities hold for any $A>0$ and $C>0$ taking
	$B$ large enough. Then, in this case, we also obtain global solution by comparison. Moreover
	all the solution of \eqref{rds} are global.		  	
    \end{proof}
	
\section{blow-up solutions}\label{BUS}
     Here we prove that there is a blow-up solution
    (the second item of Theorem  \ref{teo.1.1}).

    \begin{lemma}\label{lema.aux1}
	Let $\lambda_1,\mu_1$ be the first Direichlet eigenvalues of $\mathcal{L}$ and $\Delta$
	respectively, and $(u,v)$ be a solution of \eqref{rds}--\eqref{id}. Then there is a positive
	constant $C$ such that
	\[
	    u(x,t)\ge Ce^{-\lambda_1 t},\quad v(x,t)\ge Ce^{-\mu_1 t}\psi(x)\quad x\in\overline{\Omega}, t>0
	\]
	where $\psi$ is a positive eigenfunction of $\Delta$ associated to $\mu_1.$
    \end{lemma}
    \begin{proof}
        Let $\phi$ and  $\psi$ be positive eigenfunctions of $\mathcal{L}$ and $\Delta$
	associate to $\lambda_1$ and $\mu_1,$ respectively. It is well known that
	$\phi(x)\ge\delta>0$ in $\overline{\Omega},$ see for instance \cite{LibroJulio}.
	Then it is easy to see that there is a positive constant $k$ so small
	that $(k\phi,k\psi)$ is a subsolution of \eqref{rds}--\eqref{id}. Hence, by comparison,
	\[
	    u(x,t)\ge ke^{-\lambda_1 t}\phi(x)\ge k\delta e^{-\lambda_1 t},
	    \quad v(x,t)\ge ke^{-\mu_1 t}\psi(x)\quad x\in\overline{\Omega}, t>0.
	\]
	This completes the proof.
    \end{proof}
	
    \begin{lemma}\label{lema.bum}
	There is a blow-up solution either $\max\{\alpha,\beta\}>1$ or else
		$\max\{\alpha,\beta\}<1$ and $\kappa=pq-(1-\alpha)(1-\beta)>0.$
    \end{lemma}
	
    \begin{proof} We split the proof in three cases.
		
        \noindent {\it Case 1:} $\alpha >1.$ Let $(u,v)$ be a solution of \eqref{rds}--\eqref{id}. Then,
	by Lemma \ref{lema.aux1}, there is a positive constant $C$ such that
	\[
	    u_t(x,t)\ge Cu^\alpha(x,t) e^{-\mu_1pt}\psi_1(x)-u(x,t)\quad x\in\Omega, t>0,
	\]
	where $\mu_1$ is the first Direichlet eigenvalue of $\Delta$ and $\psi$ is a positive eigenfunction
	associated to $\mu_1.$ Then
	\[
	    e^t u_t(x,t)+e^tu(x,t)\ge C(e^t u(x,t))^{\alpha}e^{(1-\mu_1p-\alpha)t}\psi_1^p(x)
	    \quad x\in\Omega, t>0.
	\]
	Taking $w(x,t)=e^tu(x,t),$ we get
	\[
	    \dfrac{w_t(x,t)}{w^{\alpha}(x,t)}\ge Ce^{(1-\mu_1p-\alpha)t}\psi_1^p (x) \quad x\in\Omega, t>0,
	\]
	and integrating the last inequality, we deduce
	\[
	    \dfrac1{1-\alpha}\left(w^{1-\alpha}(x,t)-w^{1-\alpha}(x,0)\right)\ge\dfrac{C}{1-\mu_1p-\alpha}
			e^{(1-\mu_1p-\alpha)t}\psi_1^p(x) \quad x\in\Omega, t>0.
	\]
	Now, using $\alpha>1$ and $w(x,t)=e^tu(x,t),$ we have that
	\[
	    u(x,t)\ge e^{-t}\left[u_0^{1-\alpha}(x)-\dfrac{C(\alpha-1)}{\alpha+\mu_1p-1}
			\left(1-e^{(1-\mu_1p-\alpha)t}\right)\psi_1^p(x)\right]^{-\frac1{\alpha-1}}
			\, x\in\Omega, t>0.
	\]
	Therefore, if there is $x_0\in\Omega$ such that
	\[
	    u_0(x_0)>\left(\dfrac{C(\alpha-1)}{\alpha+\mu_1p-1}\psi_1^p(x_0)\right)^{-\frac1{\alpha-1}},
	\]
	then $u$ blows up at finite time.
		
	\bigskip
		
	\noindent {\it Case 2:} $\beta >1.$ As before, given $(u,v)$ a solution of \eqref{rds}--\eqref{id},
	 by Lemma \ref{lema.aux1}, there is a positive constant $C$ such that
	 \[
	    v_t(x,t)\ge \Delta v(x,t)+Cv^\beta(x,t) e^{-\lambda_1qt}\quad x\in\Omega, t>0,
	 \]
        where $\lambda_1$ is the first Direichlet eigenvalue of $\mathcal{L}$.
        Now, we apply Kapplan's method to prove blow-up. To do that we multiply by $\psi_1$ and integrate in $\Omega$ to obtain that
        $$
            I(t)=\int_\Omega u(x,t)\psi_1(x)dx
        $$
        satisfies
        $$
            I'\ge -\mu_1 I + Ce^{-\lambda_1 q t}I^\beta.
        $$
        Solving this inequality
        $$
            I^{1-\beta}(t)\le e^{\mu_1t} \left(I^{1-\beta}(0)-\frac{\beta-1}{\lambda_1q+\mu_1(\beta-1)}
            \left(1-e^{-(\lambda_1q+\mu_1(\beta-1))t}\right)\right)
        $$
        Therefore, as $\beta>1$ we have that for $I(0)$ large enough the function $I(t)$ blows up in finite time. Which implies that $v$ also blows up.

	\noindent {\it Case 3:} $\max\{\alpha,\beta\}<1$ and $\kappa >0.$
	In this case, the existence of blow-up solution are obtained by comparison.
	So, for a fix point $x_0\in\Omega$, we can consider the system \eqref{rds}  in a ball
	$B_L=B(x_0,L)\subset\subset\Omega$. For the component $u$ we define
	$$
	    \underline u=\varphi(t) \chi_R(x),
	$$
	where $\chi_R$ is the characteristic function in $B_R$ with $R<L$ and
	the function $\varphi$ is given by
	$$
	    \varphi(t)=(T-t)^{-2/q} e^{\frac{1-\beta}{q}(T-t)^{-1}}.
	$$
	Let us observe that the function $\psi(t)=e^{(T-t)^{-1}}$ solves the equation
	$$
	    \psi'=\varphi^q \psi^\beta.
	$$

	On the other hand, for $v$ component we consider $\underline v=w$ the solution of
	\begin{equation}\label{eq.w}
	    \left\{\begin{array}{ll}
		w_t=\Delta w+\varphi^q \chi_R w^\beta \qquad & x\in B_L, t\in(0,T),\\
				w=0 & x\in\partial B_L, t>0.
	    \end{array}\right.
	\end{equation}
	Therefore, $(\underline u, \underline v)$ is a subsolution of our problem if
	in the ball $B_R$ the following inequality holds
	$$
	    \varphi'\le \varphi \left(\int_{B_R} J(x-y)dy-1\right)+\varphi^\alpha w^p.
	$$

	We claim that there exists $\varepsilon>0$ such that the solution of \eqref{eq.w} satisfies
	\begin{equation}\label{eq.claim}
		w(x,t)\ge \varepsilon e^{(1-\varepsilon)^2 (T-t)^{-1}}
			\qquad x\in \overline{B_R}, t\in(0,T).
	\end{equation}
	Assuming this claim we have that $(\underline u, \underline v)$ is a subsolution of our problem
	if in the ball $B_R$ the following inequality holds
	$$
	    \varphi'\le \varphi \left(\int_{B_R} J(x-y)dy-1\right)+\varphi^\alpha
			\varepsilon^pe^{(1-\varepsilon)^2p (T-t)^{-1}},
	$$
	that is, if
	$$
	    \begin{array}{rl}
		\frac2q(T-t)^{-1}+
		\frac{1-\beta}{q}(T-t)^{-2} \le &\displaystyle\int_{B_{R}} J(x-y)dy-1\\
		&+(T-t)^{-\frac{2(1-\alpha)}{q}}
			e^{\left((1-\varepsilon)^2p-\frac{(1-\alpha)(1-\beta)}{q}\right)(T-t)^{-1}}.
	    \end{array}
	$$
	Since $\kappa>0,$ we can take $\varepsilon$ small enough such that
	$(1-\varepsilon)^2p-\frac{(1-\alpha)(1-\beta)}{q}>0$.
	Then, the previous inequality trivially holds taking
	$T$ small enough.

        \medskip

        In order to prove the claim, we construct two blow-up subsolutions. First, let $\phi_R(x)$ be the first
        eigenfunction of the laplacian in $B_R$ with homogeneous Dirichlet boundary condition normalize to
        $\|\phi_R\|_\infty=1$, then it is easy to see that the function
        $$
            \underline w =A \psi(t) \phi_R(x).
        $$
        extended by zero, it is a subsolution of \eqref{eq.w} provided that $A$ is small enough to have
        $$
            A^{1-\beta} \left(1+\lambda_1 (T-t)^2\right)\le 1.
        $$
        Even more, using this lower bound we get that $w$ is a supersolution of the problem
        $$
	    \left\{\begin{array}{ll}
				z_t=\Delta z \qquad & x\in B_L\setminus B_{R-\delta}, t\in (0,T),\\
				z=C e^{(T-t)^{-1}} & x\in\partial B_{R-\delta}, t\in(0,T),\\
				z=0 & x\in\partial B_L, t\in(0,T).
			\end{array}\right.	
        $$
	But, for this problem it is easy to prove that the function
	$$
	    \underline z(x,t)= C\left\{
			\begin{array}{ll}
				e^{(1-k(|x|-R+\delta))^2 (T-t)^{-1}}-1 \qquad
				& |x|\in (R-\delta,R-\delta+\frac1k), t\in(0,T),\\
				0 & |x|\in (R-\delta+\frac1k,L), t\in(0,T),
			\end{array}\right.
	$$
	is a subsolution   provided that
	$$
	    C\le A\phi_R(R-\delta) \quad \mbox{and}\quad
			k\ge \max\left\{\frac14\,,\frac{d-1}{R-\delta}\,,\frac{1}{L-R+\delta}\right\}.
	$$
	Moreover, $\underline z(R,t)>0$ if $k\delta<1$. Notice that for $\delta$ small enough, we can choose $k$
	independent of $\delta$ and by Hopf's Lemma $\phi_R(R-\delta)\sim\delta$.
	Then estimate \eqref{eq.claim} holds taking
	$\varepsilon=k\delta$.
    \end{proof}
	
\section{Non simultaneous blow-up}\label{nsu}
    In this section we study the possibility that one component of the system \eqref{rds}
    blows up at time $T<\infty$ while the other remains bounded. Thus, in this section
    we assume that the solution $(u,v)$ blows up in finite time $T$. Notice that by
    Lemma \ref{lema.bum} there exist blow-up solutions  whenever one of the following conditions holds
    $$
        \alpha>1, \quad \beta>1,\quad pq>(1-\alpha)(1-\beta).
    $$

    If only the third condition is satisfied, it is easy to see that both components blow up at the same time.
    Indeed, if we assume that $v$ is bounded, we get that $u$ is a subsolution of the problem
    $$
        \left\{
            \begin{array}{ll}
                z_t=\mathcal L z+K z^\alpha \qquad & x\in\Omega, \ t>0,\\
                z(x,t)=0& x\in\mathbb R^N\setminus\Omega, \ t>0,\\
                z(x,0)=u_0(x)&x\in\Omega.
            \end{array}\right.
    $$
    It is easy to see that for $\alpha\le 1$ and $A$ large enough $\overline z=A e^t$ is a global super-solution.
    Therefore, by comparison $u$ is global. Thus, to get non-simultaneous blow-up we need $\max\{\alpha,\beta\}>1$.

    \begin{lemma}\label{crece-en-tiempo}
        Assume that there exists $\delta_i\ge0$ such that the initial data satisfies
        \begin{equation}\label{dato-creciente}
            \begin{array}{l}
                \mathcal L u_0+ u_0^\alpha v_0^p-\delta_1 u_0^\alpha \ge0,\\
                \Delta v_0 + u_0^q v_0^\beta -\delta_2 v_0^\beta\ge 0.
            \end{array}
        \end{equation}
        Then
        $$
            u_t\ge \delta_1  u^\alpha \qquad\text{ and } \qquad  v_t\ge \delta_2  v^\beta.
        $$
    \end{lemma}
    \begin{proof}
        Let us define $\mathcal{H}=u_t-\delta_1 u^\alpha$ and $H=v_t-\delta_2 v^\beta$ which satisfy
        $$
            \left\{
                \begin{array}{ll}
                    {\mathcal H_t-\mathcal L \mathcal H
                        -\alpha u^{\alpha-1}v^p \mathcal H
                        -pv^{p-1}u^\alpha H\ge 0}& x\in \Omega, t\in (0,T),\\
                    H_t-\Delta H-\beta v^{\beta-1} u^q H- qu^{q-1}v^\beta \mathcal H
                    \ge 0\quad & x\in \Omega, t\in (0,T),
                \end{array}\right.
        $$
        with boundary conditions
        $$
            \left\{
                \begin{array}{ll}
                    \mathcal H=0 &  x\in \mathbb R^N\setminus\Omega, t\in (0,T),\\
                    H=0 &  x\in \partial\Omega, t\in (0,T),
                \end{array}\right.
        $$
        and initial data
        $$
            \left\{
                \begin{array}{ll}
                    \mathcal H(x,0)=\mathcal L u_0 +u_0^\alpha v_0^\beta-\delta u_0^\alpha\ge0,\\
                    H(x,0)=\Delta v_0+v_0^\beta u_0^q-\delta v_0^\beta \ge0.
                \end{array}\right.
        $$
        Then, by comparison the result follows.
    \end{proof}

    Now, given $\varepsilon>0,$ we consider two regular functions in $\Omega$, $f$ and $g$, such that:
    $$
        f(x)>0 \quad x\in\Omega_\varepsilon, \quad
        f(x)=0 \quad x\in\Omega\setminus \Omega_\varepsilon,
        \qquad \mathcal L f >0 \quad x\in \Omega\setminus\Omega_{2\varepsilon}
    $$
    and $g$ is a positive function in $\Omega$ vanishing at the boundary such that
    $$
        g(x)\ge m>0 \quad x\in \Omega_\varepsilon, \qquad
            \Delta g(x)> 0 \quad x\in \Omega\setminus\Omega_{2\varepsilon},
    $$
    where $\Omega_\varepsilon=\{x\in\Omega\,:\, d(x,\partial\Omega)>\varepsilon\}$.

    The idea to prove non-simultaneous blow-up is to consider initial data $(u_0,v_0)=(A f(x),B g(x))$
    to see that if $B$ is bounded and $A$ is large enough only $u$ blows up, on the contrary if
    $A$ is bounded and $B$ is large enough only $v$ blows up.

    \begin{lemma}\label{teo.nonsimul2}
        If $\alpha>1+q$, then there exists initial data such that
        $u$ blows up while $v$ remains bounded.
    \end{lemma}

    \begin{proof}
        Let $(u,v)$ be the solution of \eqref{rds} with initial data $(A f,g)$.
        In order to apply the previous lemma, we note that for
        $x\in \Omega\setminus \Omega_{2\varepsilon}$ (near the boundary)
        $$
            \Delta v_0 + u_0^q v_0^\beta =\Delta g + A^q f^q g^\beta \ge \Delta g > 0
        $$
        while for $x\in\Omega_{2\varepsilon}$ (in the interior)
        $$
            \Delta v_0 + u_0^q v_0^\beta =\Delta g + A^q \gamma^q m^\beta>0
        $$
        for $A$ large enough. On the other hand, for $x\in\Omega\setminus\Omega_\varepsilon$
        $$
            \mathcal L u_0+ u_0^\alpha v_0^p-\delta_1 u_0^\alpha = J*u_0\ge\mu >0
        $$
        while for $x\in\Omega_\varepsilon$
        $$
            \begin{array}{rl}
                \mathcal L u_0+ u_0^\alpha v_0^p-\delta_1 u_0^\alpha = & A J*f- A f +
                A^{\alpha} f^{\alpha} (g^p-\delta_1) \\
                \ge & A J*f- A f +
                A^{\alpha} f^{\alpha} (m^p-\delta_1)\\
                \ge &{AJ*f -C_\alpha (m^p-\delta_1)^{\frac{-1}{\alpha-1}}}>0
            \end{array}
        $$
        for $\delta_1<m^p$ and $A$ large.
        Applying now Lemma \ref{crece-en-tiempo}, we get
        $$
            u_t\ge \delta_1 u^\alpha, \qquad \text{ and }\qquad v_t\ge0.
        $$
        Hence, $u$ blows up a finite time $T_A$, which becomes small if $A$ is large, and
        $$
            u\le \left(\delta_1 (\alpha-1) (T_A-t)\right)^{\frac{-1}{\alpha-1}}.
        $$
        Using this inequality in the equation of $v$ we have that $v$ is a subsolution of
        $$
            z_t=\Delta z+C (T-t)^{\frac{-q}{\alpha-1}} z^\beta.
        $$
        Now, we compare with the flat supersolution
        $$
            \overline z(t)=\left\{
                \begin{array}{ll}
            \displaystyle\left(\|z_0\|_\infty^{1-\beta}+\frac{(\alpha-1)(1-\beta) C}{\alpha-1-q} (T_A^{\frac{\alpha-1-q}{\alpha-1}}-
            (T_A-t)^{\frac{\alpha-1-q}{\alpha-1}})\right)^{\frac{1}{1-\beta}} & \beta\ne1,\\
            \exp{\displaystyle\left(\ln(\|z_0\|_\infty)+\frac{(\alpha-1)(1-\beta) C}{\alpha-1-q} (T_A^{\frac{\alpha-1-q}{\alpha-1}}-
            (T_A-t)^{\frac{\alpha-1-q}{\alpha-1}})\right) }& \beta=1,
        \end{array}\right.
        $$
        to obtain that $v$ is bounded if either, $\beta\le1$, or, $\beta>1$ and $T_A$ small enough, which implies $A$ large.
    \end{proof}

    \begin{lemma}
        If $\beta>1+p$, then there exists
        initial data such that $v$ blows up while $u$ remains bounded.
    \end{lemma}
    \begin{proof}
        Let $(u,v)$ be the solution of \eqref{rds} with initial data $(f,B g)$ and $(u_l,v_l)$ the solution of the linear problem
        $$
            (u_l)_t=\mathcal L u_l, \qquad (v_l)_t=\Delta v_l
        $$
        with initial datum $(f,g)$. As, it is well known that $u_l\ge C(t)$ for $x\in\overline\Omega$ we can apply comparison to get that
        $$
            u(x,t)\ge u_l(x,t)\ge C(t), \qquad v(x,t)\ge B v_l(x,t).
        $$
        Moreover, by continuity,  we can take  $t_1$ small enough to have $\Delta v(x,t_1)\ge 0$ and $\mathcal L u\ge0$ in
        $\Omega\setminus{\Omega_{2\varepsilon}}$. So there is $\delta_2>0$ such that at time $t_1$
        $$
            \begin{array}{l}\displaystyle
                \Delta v + u^q v^\beta -\delta_2 v^\beta \ge \Delta v + (C^q(t_1)-\delta_2) B^{\beta} v_l^\beta,\\
                \displaystyle \mathcal L u+u^\alpha v^p \ge \mathcal L u +C(t_1) B v_l^p,
            \end{array}
        $$
        and both are non-negative for $\delta_2<C(t_1)$ and $B$ large. Then, by Lemma \ref{crece-en-tiempo}
        $$
            v_t\ge \delta_2 v^\beta, \qquad u_t\ge0.
        $$
        Arguing as before, we have that if $B$ is large only the $v$ component blows up.
    \end{proof}

    \begin{remark}
        Observe that by the continuity of blow-up time with respect to the initial data, there exists $A$
        and $B$ such that the solution of \eqref{rds} with initial data
        $(Af,Bg)$ blows up simultaneously provided that $\alpha>1+q$ and $\beta>1+p$.
    \end{remark}

    \begin{lemma}
        Assume that \eqref{dato-creciente} holds with $\delta_2>0$. If $v$ blows up at time $T<\infty$ and $u$ remains bounded, then $\beta>p+1$ and
        $$
            \|v(\cdot,t)\|_\infty \sim (T-t)^{\frac{-1}{\beta-1}}.
        $$
    \end{lemma}
    \begin{proof}
        Let us define $x(t)$ as $v(x(t),t)=\max_{x\in\overline\Omega} v(\cdot,t)$,
        which is a Lipschitz continuous function. Then $\partial v(x(t),t)$ exists
        almost everywhere and satisfies that $\partial v(x(t),t)=v_t(x(t),t)$.
        On the other hand, by Lemma \ref{crece-en-tiempo}
        $$
            v_t(x(t),t)\ge \delta_2 v^\beta(x(t),t).
        $$
        Integrating this inequality we get the upper blow up rate
        $$
            \|v(\cdot,t)\|_\infty \le \left(\delta_2 (\beta-1)(T-t)\right)^{\frac{-1}{\beta-1}}.
        $$
        To obtain the lower bound, we observe that as $u$ is bounded the component $v$ is a subsolution of the equation $w_t=\Delta w+C w^\beta$ and
        $$
            w(t)= \left(C (\beta-1) (T-t)\right)^{\frac{-1}{\beta-1}}
        $$
        is a solution which blows up at the same time. Assume by contradiction that there is $t_0\in(0,T)$ so that
        $$
            \max_{x\in\overline{\Omega}}v(x,t_0)< w(t_0)
        $$
        Therefore
        $$
            \max_{x\in\overline{\Omega}}v(x,t_0)< w(t_0-\varepsilon)
        $$
        for some $\varepsilon>0$. Then, by comparison, we have that
        $$
            v(x,t)\le w(t-\varepsilon)\quad\forall (x,t)\in\Omega\times (t_0,T).
        $$
        It follows that $v$ is bounded, a contradiction.

        Finally to prove that $\beta >p+1$ we observe that as $v$ is increasing in time we have that $ v(x(t_0),t)\ge v(x(t_0),t_0)$.
        Putting this estimate into the equation of $u$ we get
        \begin{equation}\label{eq.vbuena}
            u_t(x(t_0),t)\ge u^\alpha(x(t_0),t) v^{p}(x(t_0),t_0)-u(x(t_0),t).
        \end{equation}
        We can integrate this expression to obtain
        $$
            \int_{e^{t_0} u(x(t_0),t_0)}^{e^T u(x(t_0),T)} s^{-\beta} \,ds\ge
            v^p(x(t_0),t_0) \int_{t_0}^T e^{(1-\alpha) t }\,dt.
        $$
        As $u$ is a continuous bounded function, the left hand side goes to zero as $t_0$ goes to $T$. On the other hand, using  the lower blow-up rate we have that  the right hand side satisfies
        $$
            v^p(x(t_0),t_0) \int_{t_0}^T e^{(1-\alpha) t }\,dt \ge
            \widetilde C (T-t_0)^{1-\frac{p}{\beta-1}}.
        $$
        Therefore, $\beta>p+1$.
    \end{proof}

    \begin{lemma}
        Assume that \eqref{dato-creciente} holds with $\delta_1>0$. If $u$ blows up at time $T<\infty$ and $v$ remains bounded, then
        $$
        \|u(\cdot,t)\|_\infty\sim C (T-t)^{\frac{-1}{\alpha-1}}.
        $$
    \end{lemma}
    \begin{proof}
        We argue as before. The upper bound follows by Lemma \ref{crece-en-tiempo}
        and the lower bound from the fact that $v$ is bounded, which implies that $u$ is
        a sub-solution of the equation $w_t=\mathcal L w+C w^\alpha$ and
        $$
            w(t)=e^{-at}\left((\alpha-1) C (T-t)\right)^{-\frac1{\alpha-1}},
            \qquad a=1-\max_{x\in\overline{\Omega}}\int_{\Omega}J(x-y)dy
        $$
        is a solution which blows up at the same time.
    \end{proof}

\section*{Acknowledgements}L.M.D.P. was partially
supported by the European Union’s Horizon 2020 research and innovation program
under the Marie Sklodowska-Curie grant agreement No 777822,
(Agencia Nacional de Promoción de la Investigación, el Desarrollo
Tecnológico y la Innovación PICT-2018-3183, and
PICT-2019-00985 and UBACYT 20020190100367.

\end{document}